\documentclass[a4paper,reqno]{amsart}
\usepackage{geometry}
\usepackage[T1]{fontenc}
\usepackage[utf8]{inputenc}
\usepackage{lmodern}
\usepackage{amsmath,amssymb,amsthm,mathtools}
\usepackage{microtype}
\usepackage{enumitem}
\usepackage{xcolor}
\usepackage[colorlinks]{hyperref}
\usepackage[nameinlink,capitalize,noabbrev]{cleveref}
\hypersetup{
  pdftitle={The Auslander--Reiten Conjecture for Igusa-Todorov Algebras},
  pdfauthor={Xiaojin Zhang and Panyue Zhou},
  linkcolor=blue,
  citecolor=blue,
  urlcolor=blue
}
\numberwithin{equation}{section}
\setlist[enumerate]{label=\textup{(\arabic*)},itemsep=3pt,topsep=5pt}

\newtheorem{thm}{Theorem}[section]
\newtheorem{cor}[thm]{Corollary}
\newtheorem{lem}[thm]{Lemma}
\newtheorem{prop}[thm]{Proposition}
\theoremstyle{definition}
\newtheorem{defn}[thm]{Definition}
\newtheorem{rem}[thm]{Remark}
\newtheorem{exm}[thm]{Example}
\crefname{thm}{Theorem}{Theorems}
\crefname{cor}{Corollary}{Corollaries}
\crefname{lem}{Lemma}{Lemmas}
\crefname{prop}{Proposition}{Propositions}
\crefname{defn}{Definition}{Definitions}
\crefname{rem}{Remark}{Remarks}
\crefname{exm}{Example}{Examples}
\Crefname{thm}{Theorem}{Theorems}
\Crefname{cor}{Corollary}{Corollaries}
\Crefname{lem}{Lemma}{Lemmas}
\Crefname{prop}{Proposition}{Propositions}
\Crefname{defn}{Definition}{Definitions}
\Crefname{rem}{Remark}{Remarks}
\Crefname{exm}{Example}{Examples}

\DeclareMathOperator{\Ext}{Ext}
\DeclareMathOperator{\Hom}{Hom}
\DeclareMathOperator{\End}{End}
\DeclareMathOperator{\add}{add}
\DeclareMathOperator{\rad}{rad}
\DeclareMathOperator{\pd}{pd}
\DeclareMathOperator{\gldim}{gl.dim}
\DeclareMathOperator{\repdim}{rep.dim}
\DeclareMathOperator{\rank}{rank}
\DeclareMathOperator{\im}{Im}
\newcommand{\modA}{A\text{-}\mathrm{mod}}
\newcommand{\stHom}{\underline{\Hom}}
\newcommand{\perpA}{{}^{\perp}A}
\newcommand{\len}{\ell_R}
\usepackage{setspace}
\begin{document}
\title[The Auslander-Reiten conjecture for Igusa-Todorov algebras]
{The Auslander-Reiten conjecture for Igusa-Todorov algebras}
\author[Xiaojin Zhang and Panyue Zhou]{Xiaojin Zhang and Panyue Zhou}

\makeatletter
\@namedef{subjclassname@2020}{\textup{2020} Mathematics Subject Classification}
\makeatother
\date{}


\begin{abstract}
We prove that every Igusa--Todorov Artin algebra satisfies the Auslander--Reiten conjecture. More precisely, let $V$ be an $n$-Igusa--Todorov witness, and let $t$ be the number of isomorphism classes of nonprojective indecomposable summands of $V$. If a finitely generated module $M$ satisfies $\Ext_A^i(M,A)=0$ for every $i>0$ and $\Ext_A^q(M,M)=0$ for $1\leq q\leq 2t+1$, then $M$ is projective. As applications, algebras of representation dimension at most three and algebras satisfying $J^{2m+1}=0$ for which $A/J^m$ has finite representation type satisfy the Auslander--Reiten conjecture.
\end{abstract}

\maketitle

\section{Introduction}

Let $A$ be an Artin algebra over a commutative Artinian ring $R$. We write $\modA$ for the category of finitely generated left $A$-modules, and all modules considered below belong to this category. For a module $V$, let $\add V$ denote the full subcategory of direct summands of finite direct sums of copies of $V$. The Auslander--Reiten conjecture (ARC), formulated by Auslander and Reiten \cite{AR1975}, asserts that
\begin{equation}\label{eq:ARC}
\Ext_A^i(M,M\oplus A)=0\hspace{2mm}\text{for every }i>0
\quad\Longrightarrow\quad M\text{ is projective}.
\end{equation}
A module $M$ is \emph{self-orthogonal} if $\Ext_A^i(M,M)=0$ for every $i>0$. It is a \emph{generator} if $A\in\add M$. Thus \eqref{eq:ARC} is equivalent to the statement that every self-orthogonal generator is projective. The generalized Nakayama conjecture asserts that every indecomposable injective left $A$-module occurs as a direct summand of some term in a minimal injective resolution of the left regular module ${}_AA$. Auslander and Reiten \cite{AR1975} proved that the ARC holds for all Artin algebras if and only if the generalized Nakayama conjecture does. But for a given algebra, there is no such implication. The finitistic dimension conjecture, formulated by Bass \cite{Bass1960}, states that the supremum of the finite projective dimensions of modules in $\modA$ is finite. Chen, Hu, Qin and Wang \cite[Introduction]{CHQW2023} recall that its validity for all Artin algebras would imply the ARC for all Artin algebras. But this global implication does not give the corresponding implication for a fixed algebra. For further information on the interrelations among homological conjectures, see \cite{EM2026}.

The ARC has a long history of investigation. In the representation theory setting, Auslander and Reiten \cite{AR1975} proved the conjecture for algebras of finite type and algebras with radical square zero. Hoshino \cite{Hoshino1984} established the conjecture for symmetric Artin algebras with radical cube zero. The conjecture was further established for local Artin algebras with radical cube zero \cite{Xu2013}, special biserial algebras with radical cube zero \cite{Xu2015}, and a class of left Gorenstein algebras \cite{BMS2015}. Moreover, the conjecture was shown to be invariant under derived equivalences \cite{Wei2012}. For recent developments concerning the ARC for Artin algebras, we refer to Chen and Xi \cite{ChenXi2025}, Xia \cite{Xia2026}, and Zhang and Zhou \cite{ZZ2026}. In the commutative setting, Auslander, Ding and Solberg \cite{ADS1993} extended the problem to commutative Noetherian rings and established its validity for complete intersection local rings. Further contributions were made by Huneke, Şega and Vraciu \cite{HSV2004}, as well as Huneke and Leuschke \cite{HL2004}, who obtained significant vanishing results over Cohen–Macaulay and Gorenstein rings. Araya \cite{Araya2009} showed that, for commutative Gorenstein rings, the validity of the conjecture can be reduced to the codimension-one case. For more details of the ARC for Noetherian rings, we refer to \cite{CH2010}.

On October 6, 2026, OpenAI \cite{OpenAI2026} released a preprint claiming to provide a counterexample to the ARC, in which the present paper is cited among the previously established positive results. This development further highlights the importance of identifying classes of algebras for which the ARC holds. Throughout the paper, we focus on the ARC for Artin algebras.

On the other hand, Igusa and Todorov \cite{IT2005} introduced their functions to study the little finitistic dimension. Wei \cite[Theorem~2.3]{Wei2009} subsequently introduced Igusa-Todorov algebras and proved that they have finite finitistic dimension. Their defining property is the existence of short presentations of syzygies by modules from one fixed additive subcategory. This class contains syzygy-finite algebras, as discussed by Wei \cite{Wei2009} and Wei \cite[Section~5]{Wei2011}. Wei \cite[Lemma~2.1]{Wei2008} also showed that an algebra of representation dimension at most three is $0$-Igusa--Todorov. In Proposition~\ref{prop:radical-IT}, we construct a $1$-Igusa--Todorov witness under a condition on radical powers. In particular, this construction covers algebras with radical cube zero.

Further developments include the broader class of Lat-Igusa--Todorov algebras introduced by Bravo, Lanzilotta, Mendoza and Vivero \cite{BLMV2021}. Barrios, Lanzilotta and Mata \cite{BLM2024} survey the corresponding homological functions and related classes. These results do not deduce ARC from the Igusa--Todorov property alone. In particular, finite finitistic dimension for one algebra does not imply the ARC for the algebra by any known general result.

The \emph{nonprojective support} of a module $V$ is the set of isomorphism classes of its nonprojective indecomposable direct summands. We write $t(V)$ for the cardinality of this set. Thus multiplicities do not enter $t(V)$. We also write
\[
\perpA=\{X\in\modA\mid\Ext_A^i(X,A)=0\text{ for every }i>0\}.
\]
Recall that $A$ is $n$-Igusa--Todorov if there is a fixed module $V$ such that each $\Omega^nY$ admits an exact sequence $0\to V_1\to V_0\to\Omega^nY\to0$ with $V_0,V_1\in\add V$, see Definition~\ref{def:IT}. The module $V$ is called a witness. Our first result is a finite test for projectivity.

\vskip 5pt
\noindent\textbf{Theorem I}.\;{\rm (see Theorem~\ref{thm:finite-test})} \emph{Let $A$ be an $n$-Igusa--Todorov Artin algebra with witness $V$, and put $t=t(V)$. If $M\in\perpA$ and}
\[
\Ext_A^q(M,M)=0\quad(1\leq q\leq2t+1),
\]
\emph{then $M$ is projective. Consequently, every Igusa--Todorov Artin algebra satisfies the Auslander--Reiten conjecture.}
\vskip 5pt

To the best of our knowledge, Theorem~I is the first result that proves ARC for every Igusa--Todorov Artin algebra. It simultaneously covers syzygy-finite algebras and algebras of representation dimension at most three. Our argument uses the witness presentations themselves and gives a finite bound for the degree of a nonzero self-extension. Remark~\ref{rem:lat-IT} explains why the proof does not extend directly to Lat-Igusa--Todorov algebras.
The bound depends only on the nonprojective support of the chosen witness and not on $n$. No optimality is asserted. The hypothesis $M\in\perpA$ still requires vanishing in every positive degree. The finite test concerns self-extensions only.

In Proposition~\ref{prop:segment}, we prove the underlying count for a finite segment of syzygies. We pull back each projective cover extension along a presentation in $\add V$. Vanishing pullbacks force the next syzygies into $\add V$. Orthogonality in stable Hom then makes their nonprojective supports disjoint. Nonvanishing pullbacks give a triangular matrix of Ext lengths with positive diagonal. Its factorization through $\mathbb Z^t$ bounds its rank by $t$. Each of the two classes therefore contains at most $t$ indices, contradicting the existence of $2t+1$ such extensions. The matrix rank argument is related to that of Zhang and Zhou \cite[Lemma~3.8]{ZZ2026}. The entries used here are differences of Ext lengths attached to witness presentations.

By \cite[Lemma~2.1]{Wei2008}, an algebra of representation dimension at most three is $0$-Igusa--Todorov. Theorem~I supplies the new step from this presentation property to ARC and gives the following consequence.

\vskip 5pt
\noindent\textbf{Theorem II}.\; {\rm  (see Proposition~\ref{prop:repdim} )}~ \emph{Every Artin algebra of representation dimension at most three satisfies the Auslander--Reiten conjecture.}
\vskip 5pt

The second application concerns a condition on powers of the radical. We construct an explicit $1$-Igusa--Todorov witness under this condition.

\vskip 5pt
\noindent\textbf{Theorem III}.\; {\rm (see Proposition~\ref{prop:radical-IT} and Corollary~\ref{cor:radical})}~ \emph{Let $A$ be an Artin algebra with radical $J$. If $J^{2m+1}=0$ and $A/J^m$ has finite representation type for some $m\geq1$, then $A$ satisfies the Auslander--Reiten conjecture. In particular, every Artin algebra with radical cube zero satisfies the conjecture.}
\vskip 5pt

The same condition on radical powers in Theorem III appeared in the work of Dr\"axler and Happel \cite{DH1992} on the generalized Nakayama conjecture. Here we obtain ARC for arbitrary Artin algebras satisfying that condition. For split algebras of finite dimension with radical cube zero, Zhang and Zhou \cite[Theorem~1.1]{ZZ2026} obtain a bound of $3s+1$, where $s$ is the number of simple modules. That bound uses a different parameter from $2t(V)+1$ and is not uniformly smaller for every choice of witness.

The paper is structured as follows. In Section~\ref{sec:preliminaries}, we recall the definition of Igusa--Todorov algebras and establish the syzygy identities and a lifting result used in the proof. In Section~\ref{sec:main}, we prove the estimate for a finite segment, deduce Theorem~I and give a version restricted to self-orthogonal modules. In Section~\ref{sec:applications}, we prove Theorems~II and III.

Throughout the paper, we assume that all algebras are Artin $R$-algebras, where $R$ is a commutative Artin ring. We also assume that all modules are finitely generated left modules. For an algebra $A$, denote by $A$-{\rm mod} the category of finitely generated left $A$-modules. For a module $X$, choose a projective cover $P_X\twoheadrightarrow X$ and define $\Omega X$ by the exact sequence $0\to\Omega X\to P_X\to X\to0$. Higher syzygies are defined inductively, with $\Omega^0X=X$. We write $\stHom_A(X,Y)$ for $\Hom_A(X,Y)$ modulo morphisms that factor through projective modules. If $E$ is an $R$-module of finite length, then $\len E$ denotes its composition length. In particular, all Ext groups appearing in the paper have finite $R$-length.

\section{Preliminaries}\label{sec:preliminaries}

In this section, we recall some preliminaries on Igusa--Todorov algebras and syzygies of modules.

We first recall the definition of Igusa--Todorov algebras given by Wei \cite{Wei2009}, see also \cite[Section~5]{Wei2011}.

\begin{defn}\label{def:IT}
Let $n\geq0$ be an integer. The algebra $A$ is \emph{$n$-Igusa--Todorov} if there is a module $V$ such that every module $Y$ admits an exact sequence
\[
0\longrightarrow V_1\longrightarrow V_0\longrightarrow\Omega^nY\longrightarrow0,
\quad V_0,V_1\in\add V.
\]
We call $V$ an \emph{$n$-Igusa--Todorov witness}. An algebra is \emph{Igusa--Todorov} if it is $n$-Igusa--Todorov for some $n$.
\end{defn}

Adjoining $A$ to a witness preserves the Igusa--Todorov property and does not change $t(V)$.

We also need the following three lemmas on basic homological properties of syzygy modules in $\perpA$. For completeness, we provide brief proofs.

\begin{lem}\label{lem:syzygy}
Let $X\in\perpA$. Then we have the following conclusions.
\begin{enumerate}
\item $\Omega^aX\in\perpA$ for every $a\geq0$.
\item For $a\geq0$ and $q\geq1$,
\[
\Ext_A^q(\Omega^aX,\Omega^aX)\simeq\Ext_A^q(X,X).
\]
\item For $a\geq1$ and every module $Y$,
\[
\stHom_A(\Omega^aX,Y)\simeq\Ext_A^a(X,Y).
\]
\item If $\pd_A X<\infty$, then $X$ is projective.
\end{enumerate}
\end{lem}

\begin{proof}
The first assertion follows by dimension shifting. Let
\[
0\longrightarrow\Omega X\longrightarrow P\longrightarrow X\longrightarrow0
\]
be a projective cover sequence. Every finitely generated projective module belongs to $\add A$. Hence $X\in\perpA$ gives $\Ext_A^{>0}(X,P)=0$. Dimension shifting in the two variables now gives, for $q\geq1$,
\[
\Ext_A^q(\Omega X,\Omega X)
\simeq\Ext_A^{q+1}(X,\Omega X)
\simeq\Ext_A^q(X,X).
\]
Iteration proves (2).

For (3), the long exact sequence associated to the projective cover of $\Omega^{a-1}X$ identifies $\Ext_A^1(\Omega^{a-1}X,Y)$ with the quotient of $\Hom_A(\Omega^aX,Y)$ by the morphisms that extend to that projective cover. These morphisms factor through a projective module. Conversely, if a morphism factors as $\Omega^aX\to Q\to Y$ with $Q$ projective, its first factor extends to the cover because
\[
\Ext_A^1(\Omega^{a-1}X,Q)\simeq\Ext_A^a(X,Q)=0.
\]
The quotient is therefore stable Hom, and dimension shifting proves (3).

For (4), suppose that $\pd_A X=d>0$. Then $\Omega^dX$ is projective and
\[
\Ext_A^1(\Omega^{d-1}X,\Omega^dX)
\simeq\Ext_A^d(X,\Omega^dX)=0.
\]
The projective cover sequence of $\Omega^{d-1}X$ splits, making $\Omega^{d-1}X$ projective, contrary to the choice of $d$.
\end{proof}

\begin{lem}\label{lem:no-proj-summand}
If $X\in\perpA$, then $\Omega X$ has no nonzero projective direct summand.
\end{lem}

\begin{proof}
Write $0\to\Omega X\xrightarrow{u}P\to X\to0$ for its projective cover sequence. If a projective module $Q$ is a direct summand of $\Omega X$, the projection $\pi:\Omega X\to Q$ extends to $\widetilde\pi:P\to Q$, because $\Ext_A^1(X,Q)=0$. The composite of the inclusion $Q\to\Omega X\xrightarrow{u}P$ with $\widetilde\pi$ is $1_Q$. Thus the image of $Q$ is a direct summand of $P$ contained in $\rad P$. If $P=Q\oplus P'$ with this copy of $Q$, then $\rad P=\rad Q\oplus\rad P'$, and hence $Q=\rad Q$. Nakayama's lemma gives $Q=0$.
\end{proof}

\begin{samepage}
\begin{lem}\label{lem:finite-range}
Let $M\in\perpA$, let $N\geq1$ be an integer, and suppose that $\Ext_A^q(M,M)=0$ for $1\leq q\leq N$. Fix an integer $d\geq1$ and set $X_i=\Omega^{d+i}M$ for $i\geq0$. Then
\begin{enumerate}
\item $\Ext_A^q(X_i,X_i)=0$ for $1\leq q\leq N$;
\item $\stHom_A(X_a,X_b)=0$ if $1\leq a-b\leq N$;
\item $\Ext_A^q(X_a,X_b)=0$ if $a\geq b$, $q\geq1$, and $q+a-b\leq N$.
\end{enumerate}
\end{lem}
\end{samepage}

\begin{proof}
Part (1) is \cref{lem:syzygy}(2). For $a>b$, part (3) of that lemma gives
\[
\stHom_A(X_a,X_b)\simeq\Ext_A^{a-b}(X_b,X_b).
\]
For $a\geq b$ and $q\geq1$, dimension shifting gives
\[
\Ext_A^q(X_a,X_b)\simeq\Ext_A^{q+a-b}(X_b,X_b).
\]
These identities prove the remaining assertions.
\end{proof}

The following property of pullback is very essential in the paper.

\begin{lem}\label{lem:pullback}
Consider exact sequences
\[
0\longrightarrow K\xrightarrow{v}L\xrightarrow{p}X\longrightarrow0,
\quad
\varepsilon\colon~~~0\longrightarrow Y\xrightarrow{u}P\xrightarrow{\pi}X\longrightarrow0,
\]
where $P$ is projective. If the image of $[\varepsilon]$ in $\Ext_A^1(L,Y)$ under pullback along $p$ is zero, then
\[
L\oplus Y\simeq K\oplus P.
\]
In particular, if $K,L\in\add V$ and $A\in\add V$, then $Y\in\add V$.
\end{lem}

\begin{proof}
The pullback splits if and only if $p$ lifts to a morphism $h:L\to P$ with $\pi h=p$. There is then a morphism $g:K\to Y$ such that $ug=hv$. The sequence
\[
0\longrightarrow K\xrightarrow{(v,-g)}L\oplus Y
\xrightarrow{(h,u)}P\longrightarrow0
\]
is exact. Surjectivity follows by first lifting the image in $X$ through $p$. The description of the kernel follows from $\ker p=\im v$ and $\ker\pi=\im u$. It splits because $P$ is projective. The last assertion follows since $\add V$ is closed under direct summands.
\end{proof}

\section{The finite self-extension test}\label{sec:main}

In this section, we show the ARC for Igusa--Todorov algebras by giving a finite self-extension test.

The following estimate contains the counting argument. It requires presentations only for the specified syzygies.

\begin{prop}[A bound for a finite segment of syzygies]\label{prop:segment}
Let $M\in\perpA$ be nonprojective, let $V$ be a module, and let $d,N\geq1$ be integers. Suppose that
\[
\Ext_A^q(M,M)=0\quad(1\leq q\leq N).
\]
Set $X_i=\Omega^{d+i}M$ for $0\leq i\leq N$. If, for every $0\leq i<N$, there is an exact sequence
\begin{equation}\label{eq:presentation}
0\longrightarrow K_i\longrightarrow L_i\xrightarrow{p_i}X_i\longrightarrow0,
\quad K_i,L_i\in\add V,
\end{equation}
then $N\leq2t(V)$.
\end{prop}

\begin{proof}
Replace $V$ by $V\oplus A$ and then by a basic module with the same additive closure. Write $t=t(V)$. Adding $A$ introduces only projective indecomposable summands, so this replacement preserves $t$. It also preserves the assumed presentations. Since $M$ is nonprojective, \cref{lem:syzygy}(4) gives $\pd_A M=\infty$. Every $X_i$ is consequently nonzero. Moreover, $d\geq1$ ensures that each $X_i$ is a positive syzygy of $M$. It follows from \cref{lem:syzygy,lem:no-proj-summand} that $X_i\in\perpA$ and that $X_i$ has no nonzero projective direct summand. In particular, none of the $X_i$ is projective.

For $0\leq i<N$, write the projective cover extension as
\begin{equation}\label{eq:cover}
\varepsilon_i\colon ~~~
0\longrightarrow X_{i+1}\longrightarrow P_i\longrightarrow X_i\longrightarrow0.
\end{equation}
Since $X_i$ is not projective, $[\varepsilon_i]\neq0$. Denote the pullback map by
\[
\rho_i:\Ext_A^1(X_i,X_{i+1})\longrightarrow\Ext_A^1(L_i,X_{i+1}).
\]
Partition the index set $\{0,\ldots,N-1\}$ as
\[
I=\{i\mid\rho_i([\varepsilon_i])=0\},\quad
II=\{i\mid\rho_i([\varepsilon_i])\neq0\}.
\]

\smallskip
\noindent\emph{\bf The indices in $I$.}
For $i\in I$, \cref{lem:pullback} gives
\begin{equation}\label{eq:support}
L_i\oplus X_{i+1}\simeq K_i\oplus P_i,
\quad\text{and hence }X_{i+1}\in\add V.
\end{equation}
If $i<j$ lie in $I$ and $X_{i+1}$ and $X_{j+1}$ have a common indecomposable summand $U$, the composite
\[
X_{j+1}\twoheadrightarrow U\hookrightarrow X_{i+1}
\]
has nonzero stable class. Indeed, a factorization of this composite through a projective module would, after composing with the corresponding inclusion and projection, factor $1_U$ through a projective module. This would make $U$ projective, a contradiction. But \cref{lem:finite-range}(2) gives
\[
\stHom_A(X_{j+1},X_{i+1})=0,
\quad 1\leq j-i\leq N-1.
\]
Thus the nonprojective supports of the modules $X_{i+1}$, $i\in I$, are nonempty and pairwise disjoint. By \eqref{eq:support}, they lie among the $t$ nonprojective summands of $V$. Therefore
$|I|\leq t.$

\smallskip
\noindent\emph{\bf The indices in $II$.}
Write $II=\{i_1<\cdots<i_s\}$. If $s=0$, then $|II|\leq t$ is immediate. So we assume that $s>0$. Define the integer matrix $C=(c_{ab})_{1\leq a,b\leq s}$ by
\[
c_{ab}=\len\Ext_A^1(L_{i_a},X_{i_b+1})
-\len\Ext_A^1(K_{i_a},X_{i_b+1}).
\]
For later use, applying $\Hom_A(-,Y)$ to \eqref{eq:presentation} gives the exact segment
\begin{equation}\label{eq:ext-segment}
\begin{aligned}
\Ext_A^1(X_i,Y)&\longrightarrow\Ext_A^1(L_i,Y)
\longrightarrow\Ext_A^1(K_i,Y)
&\longrightarrow\Ext_A^2(X_i,Y).
\end{aligned}
\end{equation}

We divide the proof of $|II|\leq t$ into three steps:

(1) We show that $c_{ab}=0$ if $a>b$ in $II$.

If $a>b$, then $i_a\geq i_b+1$. For $r=1,2$, we have
\[
1\leq r+i_a-(i_b+1)\leq N.
\]
Indeed, $r+i_a-(i_b+1)\leq2+(N-1)-1=N$. Dimension shifting and \cref{lem:syzygy}(2) give
\[
\Ext_A^r(X_{i_a},X_{i_b+1})
\simeq\Ext_A^{r+i_a-i_b-1}(M,M)=0
\quad(r=1,2).
\]
The middle map in \eqref{eq:ext-segment}, with $i=i_a$ and $Y=X_{i_b+1}$, is therefore an isomorphism. This proves $c_{ab}=0$.

(2) We show that $c_{aa}>0$ for any $a$ in $II$.

For the diagonal, apply $\Hom_A(X_i,-)$ to \eqref{eq:cover}. Since $X_i\in\perpA$, we obtain
\[
\Ext_A^2(X_i,X_{i+1})\simeq\Ext_A^1(X_i,X_i)=0.
\]
Consequently, \eqref{eq:ext-segment} gives the short exact sequence
\[
0\longrightarrow\im\rho_i\longrightarrow\Ext_A^1(L_i,X_{i+1})
\longrightarrow\Ext_A^1(K_i,X_{i+1})\longrightarrow0.
\]
Additivity of composition length yields
\[
c_{aa}=\len(\im\rho_{i_a})>0,
\]
where strict positivity follows from $i_a\in II$. Thus $C$ is upper triangular with positive diagonal, so
\begin{equation}\label{eq:full-rank}
\rank_{\mathbb Q}C=s.
\end{equation}

(3) We show that $s\leq t$.

Let $U_1,\cdots,U_t$ be representatives of the nonprojective indecomposable summands of $V$. By Krull--Schmidt, we may write
\[
L_i\simeq L_i^{\mathrm{pr}}\oplus\bigoplus_{\alpha=1}^tU_\alpha^{\ell_{i\alpha}},
\quad
K_i\simeq K_i^{\mathrm{pr}}\oplus\bigoplus_{\alpha=1}^tU_\alpha^{k_{i\alpha}},
\]
where $L_i^{\mathrm{pr}}$ and $K_i^{\mathrm{pr}}$ are projective. Since first Ext vanishes when its first argument is projective,
\[
c_{ab}=\sum_{\alpha=1}^t
(\ell_{i_a\alpha}-k_{i_a\alpha})\,
\len\Ext_A^1(U_\alpha,X_{i_b+1}).
\]
Hence $C=DE$, where
\[
D=(\ell_{i_a\alpha}-k_{i_a\alpha})_{a,\alpha},
\quad
E=(\len\Ext_A^1(U_\alpha,X_{i_b+1}))_{\alpha,b}
\]
have sizes $s\times t$ and $t\times s$, respectively. It follows that $\rank_{\mathbb Q}C\leq t$. Together with \eqref{eq:full-rank}, this gives $|II|=s\leq t$. The same conclusion covers $t=0$, when the factorization forces $C=0$.

Combining the two counts proves
\[
N=|I|+|II|\leq2t,
\]
as required.
\end{proof}

We now apply the estimate to an Igusa--Todorov witness and get the main result in this section.

\begin{thm}[Finite self-extension test]\label{thm:finite-test}
Let $A$ be an $n$-Igusa--Todorov Artin algebra with witness $V$, and put $t=t(V)$. If $M\in\perpA$ and
\[
\Ext_A^q(M,M)=0\quad(1\leq q\leq 2t+1),
\]
then $M$ is projective. Consequently, every Igusa--Todorov Artin algebra satisfies ARC.
\end{thm}

\begin{proof}
Suppose that $M$ is nonprojective and set $N=2t+1$. For each $0\leq i<N$, apply the $n$-Igusa--Todorov property to $Y=\Omega^{i+1}M$. This gives a presentation in $\add V$ of $\Omega^nY\simeq\Omega^{n+1+i}M$. Thus \cref{prop:segment} applies with $d=n+1$ and yields $2t+1=N\leq2t$, a contradiction. The case $t=0$ is included, since then $N=1$. Finally, the vanishing in \eqref{eq:ARC} implies the hypotheses of the finite test.
\end{proof}

As an immediate result, we have the following corollary.

\begin{cor}\label{cor:selfinjective}
Let $A$ be a self-injective Igusa--Todorov Artin algebra with witness $V$. Every nonprojective module has a nonzero self-extension in some degree between $1$ and $2t(V)+1$. In particular, every self-orthogonal $A$-module is projective.
\end{cor}

\begin{proof}
Since $A$ is injective as a left module over itself, $\perpA=\modA$. Apply \cref{thm:finite-test}.
\end{proof}

The estimate for a finite segment also gives a criterion restricted to self-orthogonal modules.

\begin{prop}[A criterion on self-orthogonal modules]\label{prop:orthogonal-IT}
Suppose that there are an integer $n\geq0$ and a module $V$ such that every module $Y$ satisfying $\Ext_A^i(Y,Y\oplus A)=0$ for all $i>0$ admits an exact sequence
\[
0\longrightarrow V_1\longrightarrow V_0\longrightarrow\Omega^nY\longrightarrow0,
\quad V_0,V_1\in\add V.
\]
Then $A$ satisfies ARC.
\end{prop}

\begin{proof}
Put $t=t(V)$. If a nonprojective module $M$ satisfied the vanishing in \eqref{eq:ARC}, then every syzygy of $M$ would satisfy the same vanishing by \cref{lem:syzygy}. In particular, the assumed presentation is available for $Y=\Omega^{i+1}M$, $0\leq i\leq2t$. Applying \cref{prop:segment} with $d=n+1$ and $N=2t+1$ gives the contradiction $2t+1\leq2t$.
\end{proof}

We end this section with a remark on Lat-Igusa--Todorov algebras.

\begin{rem}\label{rem:lat-IT}
The finite number of indecomposable isomorphism classes in $\add V$ is essential to both counts in the proof. Bravo, Lanzilotta, Mendoza and Vivero \cite[Definition~5.1]{BLMV2021} introduced Lat-Igusa--Todorov algebras, for which the additional summands may range over a subcategory with infinitely many indecomposable isomorphism classes. The argument therefore does not establish the ARC for that larger class. In particular, \cref{cor:selfinjective} retains the Igusa--Todorov hypothesis.
\end{rem}

\section{Applications}\label{sec:applications}

In this section, we give applications to algebras of representation dimension at most three and algebras satisfying $J^{2m+1}=0$ for which $A/J^m$ has finite representation type. Here $J$ denotes the Jacobson radical of $A$.

\subsection{Representation dimension at most three}

A module $G$ is a \emph{generator-cogenerator} if every indecomposable projective module and every indecomposable injective module belongs to $\add G$. Auslander \cite{Auslander1971} defined the representation dimension of $A$ as the infimum of $\gldim\End_A(G)^{\mathrm{op}}$ over all generator-cogenerators $G$. The following known consequence will be used.

\begin{prop}\label{prop:repdim}
If $\repdim A\leq3$, then $A$ is $0$-Igusa--Todorov. Consequently, $A$ satisfies ARC.
\end{prop}

\begin{proof}
Choose a generator-cogenerator $G$ such that $\gldim\End_A(G)^{\mathrm{op}}\leq3$. Apply the characterization of Wei \cite[Lemma~2.1]{Wei2008} with parameter $1$. It gives an exact sequence
\[
0\longrightarrow G_1\longrightarrow G_0\longrightarrow X\longrightarrow0
\]
for every $X\in\modA$, where $G_0,G_1\in\add G$. Hence $G$ is a $0$-Igusa--Todorov witness. The second assertion follows from \cref{thm:finite-test}.
\end{proof}

\subsection{\texorpdfstring{A criterion using radical powers}{A criterion using radical powers}}

An Artin algebra $C$ has \emph{finite representation type} if $C\text{-}\mathrm{mod}$ has only finitely many isomorphism classes of indecomposable modules. A module $V$ is an \emph{additive generator} of $C\text{-}\mathrm{mod}$ if $C\text{-}\mathrm{mod}=\add V$.

\begin{prop}\label{prop:radical-IT}
Let $J=\rad A$ and let $m\geq1$. Assume that $J^{2m+1}=0$ and $C=A/J^m$ has finite representation type. If $V$ is an additive generator of $C$-\textup{mod}, viewed as an $A$-module, then
\[
W=A\oplus V\oplus\Omega_A V
\]
is a $1$-Igusa--Todorov witness for $A$.
\end{prop}

\begin{proof}
Let $N\in\modA$ and set $X=\Omega_A N$. If $P\twoheadrightarrow N$ is its projective cover, then $X\subseteq JP$. Consequently,
\[
J^{2m}X\subseteq J^{2m+1}P=0.
\]
Put $U=J^mX$ and $Q=X/U$. Both $U$ and $Q$ are annihilated by $J^m$, so they are $C$-modules and hence belong to $\add V$.

Form the pullback of a projective cover $P_Q\twoheadrightarrow Q$ along $X\twoheadrightarrow Q$. Its middle module $E$ fits into exact sequences
\[
0\longrightarrow\Omega_AQ\longrightarrow E\longrightarrow X\longrightarrow0,
\quad
0\longrightarrow U\longrightarrow E\longrightarrow P_Q\longrightarrow0.
\]
The second sequence splits, so the first becomes
\[
0\longrightarrow\Omega_AQ\longrightarrow U\oplus P_Q
\longrightarrow\Omega_A N\longrightarrow0.
\]
Since $Q\in\add V$, there is a module $Q'$ and an integer $r\geq1$ such that $Q\oplus Q'\simeq V^r$. Taking projective covers gives
\[
\Omega_AQ\oplus\Omega_AQ'\simeq\Omega_A(V^r)\simeq(\Omega_AV)^r.
\]
Thus $\Omega_AQ\in\add(\Omega_AV)$. The two left terms therefore lie in $\add W$, as required.
\end{proof}

Immediately, we have the following corollaries.

\begin{cor}\label{cor:radical}
Under the hypotheses of \cref{prop:radical-IT}, let $t=t(A\oplus V\oplus\Omega_AV)$. If $M\in\perpA$ and $\Ext_A^q(M,M)=0$ for $1\leq q\leq2t+1$, then $M$ is projective. In particular, $A$ satisfies ARC.
\end{cor}

\begin{proof}
Apply \cref{thm:finite-test} to the witness in \cref{prop:radical-IT}.
\end{proof}

\begin{cor}\label{cor:cube-zero}
Every Artin algebra with radical cube zero satisfies ARC.
\end{cor}

\begin{proof}
Take $m=1$ in \cref{cor:radical}. The algebra $A/J$ is semisimple and therefore has finite representation type.
\end{proof}

These results prove Theorems~II and III. For $J^3=0$, one may take $V$ to be the direct sum of representatives of the simple modules. The support of the witness $A\oplus V\oplus\Omega_AV$ also involves the radicals of indecomposable projective modules. This explains why $t(V\oplus\Omega_AV)$ differs from the parameter counting simple modules in the work of Zhang and Zhou \cite{ZZ2026}.

\begin{rem}
The generalized Nakayama conjecture and ARC are equivalent as assertions over all Artin algebras. This does not give a pointwise implication from the former to the latter, as explained by Chen, Hu, Qin and Wang \cite[Introduction]{CHQW2023}. The application concerning radical powers instead follows directly from the explicit Igusa--Todorov witness.
\end{rem}

Although Igusa–Todorov algebras form a broad class of algebras satisfying the Auslander–Reiten conjecture, there also exist non-Igusa–Todorov algebras satisfying the conjecture. We end the paper with the following example on non-Igusa-Todorov algebras satisfying the Auslander-Reiten conjecture.

\begin{exm}
Let $k=\mathbb C$, and let $Q$ be the quiver with one vertex and three loops $\alpha$, $\beta$ and $\gamma$.
Let $I$ be the ideal of $kQ$ generated by
\[
\alpha^2,\quad\beta^2,\quad\gamma^2,\quad
\alpha\beta-\beta\alpha,\quad
\alpha\gamma-\gamma\alpha,\quad
\beta\gamma-\gamma\beta,
\]
and set $A=kQ/I$. Then

\begin{enumerate}

\item $A$ is an $8$-dimensional commutative local algebra, but not an
Igusa--Todorov algebra.
\item $A$ satisfies the Auslander--Reiten conjecture.
In fact, every finitely generated $A$-module $M$ with
$\Ext_A^2(M,M)=0$ is projective.
\end{enumerate}
\end{exm}

\noindent \textbf{Disclosure of computational assistance.}\quad Chatgpt was used for language editing and improving the clarity of the manuscript. All mathematical results, arguments, and proofs were developed and verified by the authors.

\noindent\textbf{Acknowledgements.}\quad Xiaojin Zhang is supported by the National Natural Science Foundation of China (Grant Nos. 12171207 and 12371038). Panyue Zhou is supported by the National Natural Science Foundation of China (Grant No.~12371034).

\vspace{3mm}
\noindent\textbf{Data Availability}\hspace{2mm}
Data sharing is not applicable to this article, as no datasets were generated or analysed during the current study.

\vspace{3mm}
\noindent\textbf{Conflict of Interests}\hspace{2mm}
The authors declare that they have no conflicts of interest.

\medskip
{\footnotesize
\noindent {\bf Xiaojin Zhang}\\
School of Mathematics and Statistics, Jiangsu Normal University,\\
Xuzhou 221116, Jiangsu, P.~R.~China\\
E-mail: xjzhang@jsnu.edu.cn

\vspace{5mm}
\noindent {\bf Panyue Zhou}\\
School of Mathematics and Statistics, Changsha University of Science and Technology,\\
Changsha 410114, Hunan, P.~R.~China\\
E-mail: panyuezhou@163.com.
\par}

\end{document}